\documentclass[11pt]{amsart}
\usepackage{ifthen,etoolbox,amsmath,amssymb,color,cite,mathtools,colonequals,verbatim}
\usepackage[colorlinks,pagebackref,pdftex,bookmarks=false]{hyperref}
\apptocmd{\sloppy}{\hbadness 10000\relax}{}{}

\numberwithin{equation}{section}

\newtheorem{thm}[equation]{Theorem}
\newtheorem{prop}[equation]{Proposition}
\newtheorem{lemma}[equation]{Lemma}
\newtheorem{cor}[equation]{Corollary}

\theoremstyle{definition}
\newtheorem{rmk}[equation]{Remark}
\newtheorem{notation}[equation]{Notation}
\newtheorem{defn}[equation]{Definition}

\newcommand{\F}{\mathbb{F}}
\newcommand{\bP}{\mathbb{P}}

\newcommand{\mybar}[1]{#1\llap{$\overline{\phantom{\rm#1}}$}}

\begin{document}

\title{Several classes of permutation pentanomials}

\author{Zhiguo Ding}
\address{
  School of Mathematics and Statistics, 
  Hunan Normal University, 
  Changsha 410081, China
}
\email{ding8191@qq.com}

\date{\today}

\keywords{}

\date{\today}

\begin{abstract}
For each prime $p$ and each power $q=p^k$, we present two large classes of permutation polynomials over $\F_{q^2}$ of the form $X^r B(X^{q-1})$ which have at most five terms, where $B(X)$ is a polynomial with coefficients in the prime field of $\F_{q^2}$ except at most one. 
\end{abstract}

\maketitle


\section{Introduction}

A polynomial $f(X)\in\F_q[X]$ is called a \emph{permutation polynomial} if the map $\alpha\mapsto f(\alpha)$ permutes $\F_q$. In recent years, many authors have constructed classes of permutation polynomials having at most five terms and having coefficients in $\{1,-1\}$. In particular, many of these are permutation polynomials over $\F_{q^2}$ with the form $X^r B(X^{q-1})$. 

We will present several classes of permutation polynomials with at most five terms, whose coefficients are all $\pm 1$ except at most one. In particular, we present several classes of permutation pentanomials with coefficients $\pm 1$, and several classes of permutation quadrinomials with coefficients $\pm 1$. 

The following is our first main result, which by use of roots of unity presents two general classes of permutation polynomials with at most five terms, whose coefficients are all $\pm 1$ except at most one.

\begin{thm}\label{pp-5}
Assume $q=p^k$ for some prime $p$ and some integer $k>0$. Suppose $Q,R,S$ are integers of the form $p^i$ for some integer $i\ge 0$. Let $r>0$ be an integer such that $r\equiv Q+R+S\pmod{q+1}$ and $\gcd(r,q-1)=1$. Suppose $m\ge 3$ is an integer with $q^2\equiv 1\pmod m$, and $v$ is a primitive $m$-th root $v$ of unity in\/ $\F_{q^2}$. Write $b\colonequals (1-v^4)/(v^3-v)\in\F_{q^2}$, so that $b\in\{0,\pm 1\}$ if and only if $m\in\{3,4,6\}$. For each $j\in\{1,2\}$ let $f_j(X)\colonequals X^r B_j(X^{q-1})$, where we write
\begin{align*}
B_1(X) &\colonequals X^{Q+R+S}-X^Q-X^R-X^S+b, \\
B_2(X) &\colonequals X^{Q+R}-X^{Q+S}-X^{R+S}+bX^S-1.
\end{align*}
Then the following statements hold:
\begin{enumerate}
\item $f_1(X)$ permutes\/ $\F_{q^2}$ if one of the following holds:
\begin{itemize}
\item $q\equiv 1\pmod m$ and $(Q,R,S) \equiv \pm (1,1,1) \pmod m$ and $\gcd(Q+R+S,q+1)=1$,
\item $q\equiv -1\pmod m$ and $(Q,R,S) \equiv \pm (1,1,1) \pmod m$ and $\gcd(Q+R+S,q-1)=1$,
\item $q\equiv 1\pmod m$ and $(Q,R,S) \equiv \pm (1,1,-1) \pmod m$ and $\gcd(Q+R-S,q+1)=1$;    
\end{itemize}
\item $f_2(X)$ permutes\/ $\F_{q^2}$ if one of the following holds:
\begin{itemize}
\item $q\equiv 1\pmod m$ and $(Q,R,S) \equiv \pm (1,1,-1) \pmod m$ and $\gcd(Q+R+S,q+1)=1$,
\item $q\equiv -1\pmod m$ and $(Q,R,S) \equiv \pm (1,1,-1) \pmod m$ and $\gcd(Q+R+S,q-1)=1$,
\item $q\equiv 1\pmod m$ and $(Q,R,S) \equiv \pm (1,1,1) \pmod m$ and $\gcd(Q+R-S,q+1)=1$,
\item $q\equiv 1\pmod m$ and $(Q,R,S) \equiv \pm (1,-1,1) \pmod m$ and $\gcd(Q-R-S,q+1)=1$.
\end{itemize}
\end{enumerate}
\end{thm}

Theorem~\ref{pp-5} is a direct consequence from the combination of the following two classification results on permutation pentanomials.

\begin{thm}\label{mth}
Assume $q=p^k$ for some prime $p$ and some integer $k>0$. Suppose $Q,R,S$ are integers of the form $p^i$ for some integer $i\ge 0$. Let $r>0$ be an integer with $r \equiv Q+R+S \pmod{q+1}$. Let $m\ge 3$ be an integer with $q\equiv \pm 1\pmod m$. Write $b\colonequals (1-v^4)/(v^3-v)$ for a primitive $m$-th root $v$ of unity, so that $b\in\{0,\pm 1\}$ if and only if $m\in\{3,4,6\}$. Write $f(X)\colonequals X^r B(X^{q-1})$ where one of the following holds:
\begin{enumerate}
\item $B(X)\colonequals X^{Q+R+S}-X^Q-X^R-X^S+b$ with $(Q,R,S) \equiv \pm (1,1,1) \pmod m$;
\item $B(X)\colonequals X^{Q+R}-X^{Q+S}-X^{R+S}+bX^S-1$ with $(Q,R,S) \equiv \pm (1,1,-1) \pmod m$.
\end{enumerate}
Then $f(X)$ permutes\/ $\F_{q^2}$ if and only if $\gcd(r,q-1)=1$ and one of the following holds:
\begin{enumerate}
\item $q\equiv 1\pmod m$ and $\gcd(Q+R+S,q+1)=1$;
\item $q\equiv -1\pmod m$ and $\gcd(Q+R+S,q-1)=1$.
\end{enumerate}  
\end{thm}

\begin{thm}\label{mth'}
Assume $q=p^k$ for some prime $p$ and some integer $k>0$. Suppose $Q,R,S$ are integers of the form $p^i$ for some integer $i\ge 0$. Let $r>0$ be an integer with $r \equiv Q+R+S \pmod{q+1}$. Let $m\ge 3$ be an integer with $q\equiv \pm 1\pmod m$. Write $b\colonequals (1-v^4)/(v^3-v)$ for a primitive $m$-th root $v$ of unity, so that $b\in\{0,\pm 1\}$ if and only if $m\in\{3,4,6\}$. Write $f(X)\colonequals X^r B(X^{q-1})$ where one of the following holds:
\begin{enumerate}
\item $B(X)\colonequals X^{Q+R}-X^{Q+S}-X^{R+S}+bX^S-1$ with $(Q,R,S) \equiv \pm (1,1,1) \pmod m$;
\item $B(X)\colonequals X^{R+S}-X^{Q+R}-X^{Q+S}+bX^Q-1$ with $(Q,R,S) \equiv \pm (1,-1,1) \pmod m$;
\item $B(X)\colonequals X^{Q+R+S}-X^Q-X^R-X^S+b$ with $(Q,R,S) \equiv \pm (1,1,-1) \pmod m$.
\end{enumerate}
Then $f(X)$ permutes\/ $\F_{q^2}$ if and only if $q\equiv 1\pmod m$ and $\gcd(r,q-1)=1$ and $\gcd(Q+R-S,q+1)=1$.
\end{thm}

\begin{rmk}
In light of Lemma~\ref{old}, Theorems~\ref{pp-5}, \ref{mth}, and ~\ref{mth'} remain true if $B(X)$ is replaced with $\pm X^{Q+R+S} B(1/X)$. As one can see in the proofs, permutation polynomials in Theorems~\ref{mth} come from $X^{Q+R+S}$, while permutation polynomials in Theorem~\ref{mth'} come from $X^{Q+R-S}$. 
\end{rmk}

As direct consequences, let us look at the special cases of Theorem~\ref{pp-5}, \ref{mth}, and ~\ref{mth'} in which $m\in\{3,4,6\}$, respectively. 

The $m\in\{3,6\}$ case of Theorem~\ref{pp-5} is as follows, which presents four large classes of permutation pentanomials with coefficients $\pm 1$.

\begin{cor}\label{pp-36th}
Assume $q=p^k$ for some prime $p\ne 3$ and some integer $k>0$. Suppose $Q,R,S$ are integers of the form $p^i$ for some integer $i\ge 0$. Let $r>0$ be an integer such that $r\equiv Q+R+S\pmod{q+1}$ and $\gcd(r,q-1)=1$. For each $j\in\{1,2,3,4\}$ let $f_j(X)\colonequals X^r B_j(X^{q-1})$, where we write
\begin{align*}
B_1(X) &\colonequals X^{Q+R+S}-X^Q-X^R-X^S+1, \\
B_2(X) &\colonequals X^{Q+R+S}-X^Q-X^R-X^S-1, \\
B_3(X) &\colonequals X^{Q+R}-X^{Q+S}-X^{R+S}+X^S-1,\\
B_4(X) &\colonequals X^{Q+R}-X^{Q+S}-X^{R+S}-X^S-1.
\end{align*}
Then the following statements hold:
\begin{enumerate}
\item for each $j\in\{1,2\}$, $f_j(X)$ permutes\/ $\F_{q^2}$ if one of the following holds:
\begin{itemize}
\item $q\equiv 1\pmod 3$ and $(Q,R,S) \equiv \pm (1,1,1) \pmod 3$ and $\gcd(Q+R+S,q+1)=1$,
\item $q\equiv -1\pmod 3$ and $(Q,R,S) \equiv \pm (1,1,1) \pmod 3$ and $\gcd(Q+R+S,q-1)=1$,
\item $q\equiv 1\pmod 3$ and $(Q,R,S) \equiv \pm (1,1,-1) \pmod 3$ and $\gcd(Q+R-S,q+1)=1$;    
\end{itemize}
\item for each $j\in\{3,4\}$, $f_j(X)$ permutes\/ $\F_{q^2}$ if one of the following holds:
\begin{itemize}
\item $q\equiv 1\pmod 3$ and $(Q,R,S) \equiv \pm (1,1,-1) \pmod 3$ and $\gcd(Q+R+S,q+1)=1$,
\item $q\equiv -1\pmod 3$ and $(Q,R,S) \equiv \pm (1,1,-1) \pmod 3$ and $\gcd(Q+R+S,q-1)=1$,
\item $q\equiv 1\pmod 3$ and $(Q,R,S) \equiv \pm (1,1,1) \pmod 3$ and $\gcd(Q+R-S,q+1)=1$,
\item $q\equiv 1\pmod 3$ and $(Q,R,S) \equiv \pm (1,-1,1) \pmod 3$ and $\gcd(Q-R-S,q+1)=1$.
\end{itemize}
\end{enumerate}
\end{cor}

The $m=4$ case of Theorem~\ref{pp-5} can be restated as follows, which presents two large classes of permutation quadrinomials with coefficients $\pm 1$.

\begin{cor}\label{pp-4th}
Assume $q=p^k$ for some odd prime $p$ and some integer $k>0$. Suppose $Q,R,S$ are integers of the form $p^i$ for some integer $i\ge 0$. Let $r>0$ be an integer such that $r\equiv Q+R+S\pmod{q+1}$ and $\gcd(r,q-1)=1$. For each $j\in\{1,2\}$ let $f_j(X)\colonequals X^r B_j(X^{q-1})$ where
\begin{align*}
B_1(X) &\colonequals X^{Q+R}+X^{Q+S}+X^{R+S}-1, \\
B_2(X) &\colonequals X^{Q+R}-X^{Q+S}-X^{R+S}-1.
\end{align*}
Then the following statements hold:
\begin{enumerate}
\item $f_1(X)$ permutes\/ $\F_{q^2}$ if one of the following holds:
\begin{itemize}
\item $q\equiv 1\pmod 4$ and $(Q,R,S) \equiv \pm (1,1,1) \pmod 4$ and $\gcd(Q+R+S,q+1)=1$,
\item $q\equiv -1\pmod 4$ and $(Q,R,S) \equiv \pm (1,1,1) \pmod 4$ and $\gcd(Q+R+S,q-1)=1$,
\item $q\equiv 1\pmod 4$ and $(Q,R,S) \equiv \pm (1,1,-1) \pmod 4$ and $\gcd(Q+R-S,q+1)=1$;    
\end{itemize}
\item $f_2(X)$ permutes\/ $\F_{q^2}$ if one of the following holds:
\begin{itemize}
\item $q\equiv 1\pmod 4$ and $(Q,R,S) \equiv \pm (1,1,-1) \pmod 4$ and $\gcd(Q+R+S,q+1)=1$,
\item $q\equiv -1\pmod 4$ and $(Q,R,S) \equiv \pm (1,1,-1) \pmod 4$ and $\gcd(Q+R+S,q-1)=1$,
\item $q\equiv 1\pmod 4$ and $(Q,R,S) \equiv \pm (1,1,1) \pmod 4$ and $\gcd(Q+R-S,q+1)=1$,
\item $q\equiv 1\pmod 4$ and $(Q,R,S) \equiv \pm (1,-1,1) \pmod 4$ and $\gcd(Q-R-S,q+1)=1$.
\end{itemize}
\end{enumerate}
\end{cor}

The $m\in\{3,6\}$ case of Theorem~\ref{mth} can be restated as follows.

\begin{cor}\label{36th}
Assume $q=p^k$ for some prime $p\ne 3$ and some integer $k>0$. Suppose $Q,R,S$ are integers of the form $p^i$ for some integer $i\ge 0$. Let $r>0$ be an integer with $r \equiv Q+R+S \pmod{q+1}$. Write $f(X)\colonequals X^r B(X^{q-1})$ where one of the following holds:
\begin{enumerate}
\item $B(X)\colonequals X^{Q+R+S}-X^Q-X^R-X^S+1$ with $(Q,R,S) \equiv \pm (1,1,1) \pmod 3$;
\item $B(X)\colonequals X^{Q+R+S}-X^Q-X^R-X^S-1$ with $(Q,R,S) \equiv \pm (1,1,1) \pmod 3$;
\item $B(X)\colonequals X^{Q+R}-X^{Q+S}-X^{R+S}+X^S-1$ with $(Q,R,S) \equiv \pm (1,1,-1) \pmod 3$;
\item $B(X)\colonequals X^{Q+R}-X^{Q+S}-X^{R+S}-X^S-1$ with $(Q,R,S) \equiv \pm (1,1,-1) \pmod 3$.
\end{enumerate}
Then $f(X)$ permutes\/ $\F_{q^2}$ if and only if $\gcd(r,q-1)=1$ and one of the following holds:
\begin{enumerate}
\item $q\equiv 1\pmod 3$ and $\gcd(Q+R+S,q+1)=1$;
\item $q\equiv -1\pmod 3$ and $\gcd(Q+R+S,q-1)=1$.
\end{enumerate}  
\end{cor}

The $m=4$ case of Theorem~\ref{mth} can be restated as follows.

\begin{cor}\label{4th}
Assume $q=p^k$ for some odd prime $p$ and some integer $k>0$. Suppose $Q,R,S$ are integers of the form $p^i$ for some integer $i\ge 0$. Let $r>0$ be an integer with $r \equiv Q+R+S \pmod{q+1}$. Write $f(X)\colonequals X^r B(X^{q-1})$ where one of the following holds:
\begin{enumerate}
\item $B(X)\colonequals X^{Q+R}+X^{Q+S}+X^{R+S}-1$ with $(Q,R,S) \equiv \pm (1,1,1) \pmod 4$;
\item $B(X)\colonequals X^{Q+R}-X^{Q+S}-X^{R+S}-1$ with $(Q,R,S) \equiv \pm (1,1,-1) \pmod 4$.
\end{enumerate}
Then $f(X)$ permutes\/ $\F_{q^2}$ if and only if $\gcd(r,q-1)=1$ and one of the following holds:
\begin{enumerate}
\item $q\equiv 1\pmod 4$ and $\gcd(Q+R+S,q+1)=1$;
\item $q\equiv -1\pmod 4$ and $\gcd(Q+R+S,q-1)=1$.
\end{enumerate}  
\end{cor}

The $m\in\{3,6\}$ case of Theorem~\ref{mth'} can be restated as follows.

\begin{cor}\label{36th'}
Assume $q=p^k$ for some prime $p\ne 3$ and some integer $k>0$. Suppose $Q,R,S$ are integers of the form $p^i$ for some integer $i\ge 0$. Let $r>0$ be an integer with $r \equiv Q+R+S \pmod{q+1}$. Write $f(X)\colonequals X^r B(X^{q-1})$ where one of the following holds:
\begin{enumerate}
\item $B(X)\colonequals X^{Q+R}-X^{Q+S}-X^{R+S}+X^S-1$ with $(Q,R,S) \equiv \pm (1,1,1) \pmod 3$;
\item $B(X)\colonequals X^{Q+R}-X^{Q+S}-X^{R+S}-X^S-1$ with $(Q,R,S) \equiv \pm (1,1,1) \pmod 3$;
\item $B(X)\colonequals X^{R+S}-X^{Q+R}-X^{Q+S}+X^Q-1$ with $(Q,R,S) \equiv \pm (1,-1,1) \pmod 3$;
\item $B(X)\colonequals X^{R+S}-X^{Q+R}-X^{Q+S}-X^Q-1$ with $(Q,R,S) \equiv \pm (1,-1,1) \pmod 3$;
\item $B(X)\colonequals X^{Q+R+S}-X^Q-X^R-X^S+1$ with $(Q,R,S) \equiv \pm (1,1,-1) \pmod 3$;
\item $B(X)\colonequals X^{Q+R+S}-X^Q-X^R-X^S-1$ with $(Q,R,S) \equiv \pm (1,1,-1) \pmod 3$.
\end{enumerate}
Then $f(X)$ permutes\/ $\F_{q^2}$ if and only if $q\equiv 1\pmod 3$ and $\gcd(r,q-1)=1$ and $\gcd(Q+R-S,q+1)=1$.
\end{cor}

The $m=4$ case of Theorem~\ref{mth'} can be restated as follows.

\begin{cor}\label{4th'}
Assume $q=p^k$ for some odd prime $p$ and some integer $k>0$. Suppose $Q,R,S$ are integers of the form $p^i$ for some integer $i\ge 0$. Let $r>0$ be an integer with $r \equiv Q+R+S \pmod{q+1}$. Write $f(X)\colonequals X^r B(X^{q-1})$ where one of the following holds:
\begin{enumerate}
\item $B(X)\colonequals X^{Q+R}-X^{Q+S}-X^{R+S}-1$ with $(Q,R,S) \equiv \pm (1,1,1) \pmod 4$;
\item $B(X)\colonequals X^{R+S}-X^{Q+R}-X^{Q+S}-1$ with $(Q,R,S) \equiv \pm (1,-1,1) \pmod 4$;
\item $B(X)\colonequals X^{Q+R}+X^{Q+S}+X^{R+S}-1$ with $(Q,R,S) \equiv \pm (1,1,-1) \pmod 4$.
\end{enumerate}
Then $f(X)$ permutes \/ $\F_{q^2}$ if and only if $q\equiv 1\pmod 4$ and $\gcd(r,q-1)=1$ and $\gcd(Q+R-S,q+1)=1$.
\end{cor}

In our final result, which relies on the following notion, we show that in the special case $r=Q+R+S$ the polynomials $f(X)$ in Theorems~\ref{mth} and \ref{mth'} can be expressed as compositions of functions having very simple forms.

\begin{defn}
If $U$ and $V$ are $\F_q$-vector spaces, then a function $f\colon U\to U$ is \emph{$\F_q$-linearly equivalent} to a function $g\colon V\to V$ if $f=\rho\circ g\circ\eta$ for some $\F_q$-vector space isomorphisms $\rho\colon V \to U$ and $\eta\colon U\to V$.
\end{defn}

\begin{rmk}
It is easy to see that $\F_q$-linear equivalence is an equivalence relation on the union of the sets of functions $\F_{q^2}\to\F_{q^2}$ and $\F_q\times\F_q\to\F_q\times\F_q$, and that $\F_q$-linear equivalence preserves the property of a function being bijective.
\end{rmk}

\begin{rmk}\label{iso}
It is well-known that the $\F_q$-vector space automorphisms of $\F_{q^2}$ are the functions induced by $cX^q+dX$ where $c,d\in\F_{q^2}$ satisfy $c^{q+1}\ne d^{q+1}$.
Likewise, the $\F_q$-vector space isomorphisms $\F_{q^2}\to\F_q\times\F_q$ are the functions $x\mapsto \bigl(cx+(cx)^q,dx+(dx)^q\bigr)$ where $c,d\in\F_{q^2}^*$ satisfy $c^{q-1}\ne d^{q-1}$, and the $\F_q$-vector space isomorphisms $\F_q\times\F_q\to\F_{q^2}$ are $(x,y)\mapsto cx+dy$ where $c,d\in\F_{q^2}^*$ satisfy $c^{q-1}\ne d^{q-1}$.
\end{rmk}

\begin{thm}\label{equiv}
With the notation in Theorems~\ref{mth} or ~\ref{mth'} we have:
\begin{itemize}
\item if $r=Q+R+S$ and $q\equiv 1\pmod m$ then $f(X)$ in Theorem~\ref{mth} is $\F_q$-equivalent to $X^{Q+R+S}$ on\/ $\F_{q^2}$; 
\item if $r=Q+R+S$ and $q\equiv -1\pmod m$ then $f(X)$ in Theorem~\ref{mth} is $\F_q$-equivalent to $(X^{Q+R+S},Y^{Q+R+S})$ on\/ $\F_q\times\F_q$; 
\item if $r=Q+R+S$ and $q\equiv 1\pmod m$ then $f(X)$ in Theorem~\ref{mth'} is $\F_q$-equivalent to $X^{Q+R+qS}$ on\/ $\F_{q^2}$;
\item if $r=Q+R+S$ and $q\equiv -1\pmod m$ then $f(X)$ in Theorem~\ref{mth'} is $\F_q$-equivalent to $(X^{Q+R}Y^S,Y^{Q+R}X^S)$ on\/ $\F_q\times\F_q$. 
\end{itemize}
\end{thm}

\begin{rmk}
In light of the well-known fact that $X^n$ permutes $\F_{q^i}$ if and only if $\gcd(n,q^i-1)=1$, Theorem~\ref{equiv} gives an alternative proof of the special case $r=Q+R+S$ of Theorems~\ref{mth} and \ref{mth'}.
\end{rmk}

This paper is organized as follows. In Section~\ref{fact} we give some background results. In Section~\ref{proof} we show Theorems~\ref{mth}, ~\ref{mth'}, and ~\ref{equiv}.


\section{Background results}\label{fact}

In this section we present the background results which are used in our proof of Theorems~\ref{mth} and ~\ref{mth'}. They rely on the following notation.

\begin{notation}
If $q$ is a prime power, then we write $\mu_{q+1}$ for the set of all $(q+1)$-th roots of unity in $\F_{q^2}$, and we define $\bP^1(\F_q)\colonequals \F_q\cup\{\infty\}$.
\end{notation}

We begin with the following special case of \cite[Lemma~2.1]{Zlem}.

\begin{lemma}\label{old}
Write $f(X)\colonequals X^r B(X^{q-1})$ where $q$ is a prime power, $r$ is a positive integer, and $B(X)\in\F_{q^2}[X]$. Then $f(X)$ permutes\/ $\F_{q^2}$ if and only if $\gcd(r,q-1)=1$ and $X^r B(X)^{q-1}$ permutes $\mu_{q+1}$.
\end{lemma}

The next two results are reformulations of \cite[Lemmas~2.1 and 3.1]{Z-Redei}.

\begin{lemma}\label{deg1mu}
If $\alpha,\beta\in\F_{q^2}$ satisfy $\alpha^{q+1}\ne\beta^{q+1}$, then $(\beta^qX+\alpha^q)/(\alpha X+\beta)$ 
permutes $\mu_{q+1}$.
\end{lemma}

\begin{lemma}\label{mu}
If $\alpha\in\F_{q^2}\setminus\F_q$ and $\beta\in\mu_{q+1}$, then $(\alpha X+\beta\alpha^q)/(X+\beta)$ maps $\mu_{q+1}$ bijectively onto\/ $\bP^1(\F_q)$.
\end{lemma}

The following result follows from the combination of Lemmas~\ref{deg1mu} and ~\ref{mu}.

\begin{prop}\label{inverse}
Assume $v\in\mybar\F_q$ with $v^2\ne 1$. Write $\rho(X)\colonequals (vX+1)/(X+v)$ and $\eta(X)\colonequals (-vX+1)/(X-v)$. Then $\rho(X)$ and $\eta(X)$ are degree-$1$ rational functions over\/ $\mybar\F_q$ with $\rho^{-1}(X)=\eta(X)$. Moreover, both of the following statements hold:
\begin{enumerate}
\item if $v\in\F_q$ then $\rho(X)$ permutes both\/ $\bP^1(\F_q)$ and $\mu_{q+1}$;
\item if $v\in\mu_{q+1}$ then $\rho(X)$ induces a bijection from\/ $\bP^1(\F_q)$ onto $\mu_{q+1}$ and also a bijection from $\mu_{q+1}$ onto\/ $\bP^1(\F_q)$.
\end{enumerate}
\end{prop}

\begin{proof}
If $v\in\F_q$ then $v\in\F_{q^2}$ with $v^{q+1}\ne 1$, which by Lemma~\ref{deg1mu} implies $\rho(X)$ and $\eta(X)$ permute $\mu_{q+1}$. If $v\in\mu_{q+1}$ then $v\in\F_{q^2}\setminus\F_q$ and $-v\in\mu_{q+1}$, which by Lemma~\ref{mu} implies $\rho(X)$ and $\eta(X)$ map $\mu_{q+1}$ bijectively onto $\bP^1(\F_q)$. Since $\rho(X)$ is the inverse of $\eta(X)$ under the composition, it follows that $\rho(X)$ and $\eta(X)$ maps $\bP^1(\F_q)$ bijectively onto $\mu_{q+1}$. 
\end{proof}

The following result is a direct consequence of Proposition~\ref{inverse}.

\begin{cor}\label{geo-cyc}
Assume $v\in\F_q\cup\mu_{q+1}$ with $v^2\ne 1$. Write $\rho(X)\colonequals (vX+1)/(X+v)$ and $\eta(X)\colonequals (-vX+1)/(X-v)$. Let $n>0$ be an integer and let $g(X)\colonequals \eta(X)\circ X^n\circ\rho(X)$. Then $g(X)$ induces a map from\/ $\bP^1(\F_q)$ into itself and a map from $\mu_{q+1}$ into itself, satisfying:
\begin{enumerate}
\item $g(X)$ permutes\/ $\bP^1(\F_q)$ if and only if one of the following holds:
\begin{itemize}
\item $v\in\F_q$ and $\gcd(n,q-1)=1$,
\item $v\in\mu_{q+1}$ and $\gcd(n,q+1)=1$;
\end{itemize}
\item $g(X)$ permutes $\mu_{q+1}$ if and only if one of the following holds:
\begin{itemize}
\item $v\in\F_q$ and $\gcd(n,q+1)=1$,
\item $v\in\mu_{q+1}$ and $\gcd(n,q-1)=1$.
\end{itemize} 
\end{enumerate}
\end{cor}

\begin{rmk}
The formal numerator $N(X)$ and denominator $D(X)$ of $g(X)$ in Corollary~\ref{geo-cyc} have few terms when $n$ is a sum of few powers of the characteristic $p$ of $\F_q$. More precisely, for any integer $\ell>0$, if $n$ is a sum of $\ell$ powers of $p$ then both $N(X)$ and $D(X)$ have at most $2^{\ell}$ terms. In the case that $n$ is a sum of two powers of $p$, the properties of $g(X)$ described in Corollary~\ref{geo-cyc} have been studied systematically in \cite{DZquad}. In this paper, we will study such $g(X)$ in the case that $n$ is a sum of three powers $Q,R,S$ of $p$ and $v$ is a primitive $m$-th root of unity for some integer $m\ge 3$. In this case, in general $N(X)$ and $D(X)$ have at most eight terms, and they have at most five terms under the additional assumption that $Q,R,S\equiv \pm 1\pmod m$. Moreover, if $m\in\{3,4,6\}$ and $Q,R,S$ are pairwise distinct, then all coefficients of $N(X)$ and $D(X)$ are the same up to a sign, so that we can present several classes of permutation pentanomials or quadrinomials with all coefficients $\pm 1$. In the case that $n$ is a sum of more than three powers of $p$, although our argument works, $N(X)$ and $D(X)$ have more terms which appear like less elegant than those presented in this paper. 
\end{rmk}


\section{Proofs of Theorems~\ref{mth}, ~\ref{mth'}, and ~\ref{equiv}}\label{proof}

In this section, we give proofs of Theorems~\ref{mth}, \ref{mth'}, and \ref{equiv}. First we state as follows explicitly the case of Corollary~\ref{geo-cyc} in which the integer $n$ is a sum of three powers of the characteristic of $\F_q$.

\begin{prop}\label{tool}
Assume $q=p^k$ for some prime $p$ and some integer $k>0$. Suppose $Q,R,S$ are integers of the form $p^i$ for some integer $i\ge 0$. Pick $v\in\F_q\cup\mu_{q+1}$ with $v^2\ne 1$. Write $n\colonequals Q+R+S$ and $D(X)\colonequals (vX+1)^n-v(X+v)^n$, so that we have explicitly 
\begin{align*}
D(X) = &\,
(v^{Q+R+S}-v) X^{Q+R+S}
+ (v^{Q+R}-v^{S+1}) X^{Q+R} \\
&\,+ (v^{Q+S}-v^{R+1}) X^{Q+S}
+ (v^{R+S}-v^{Q+1}) X^{R+S} \\
&\,+ (v^Q-v^{R+S+1}) X^Q
+ (v^R-v^{Q+S+1}) X^R \\
&\,+ (v^S-v^{Q+R+1}) X^S
+ (1-v^{Q+R+S+1}).
\end{align*}
Then $D(X)$ has no roots in $\mu_{q+1}$, so that $g(X)\colonequals X^n D^{(q)}(1/X) /D(X)$ induces a map from $\mu_{q+1}$ into itself, where $D^{(q)}(X)$ is the polynomial obtained from $D(X)$ by raising all coefficients to their $q$-th powers. Moreover, $g(X)$ permutes $\mu_{q+1}$ if and only if one of the following holds:
\begin{itemize}
\item $v\in\F_q$ and $\gcd(n,q+1)=1$,
\item $v\in\mu_{q+1}$ and $\gcd(n,q-1)=1$.
\end{itemize} 
\end{prop}

\begin{proof}
Write $\rho(X)\colonequals (vX+1)/(X+v)$ and $\eta(X)\colonequals (-vX+1)/(X-v)$. By Corollary~\ref{geo-cyc} $g_1(X)\colonequals \eta(X)\circ X^n\circ\rho(X)$ induces a map from $\mu_{q+1}$ into itself. Moreover, $g_1(X)$ permutes $\mu_{q+1}$ if and only if either [ $v\in\F_q$ and $\gcd(n,q+1)=1$ ] or [ $v\in\mu_{q+1}$ and $\gcd(n,q-1)=1$ ]. It is easy to check that $g_1(X) = N(X) / D(X)$ where $N(X)\colonequals X^n D(1/X)$. Since $\deg(g_1)=n$ and $\max(\deg(N),\deg(D))\le n$, it follows that $N(X)$ and $D(X)$ are coprime. Thus $D(X)$ has no roots in $\mu_{q+1}$ since $g_1(X)$ maps $\mu_{q+1}$ into itself. It is routine to verify that $g_1(X)=g(X)$ if $v\in\F_q$, and $g_1(X)=-v^{n+1}g(X)$ if $v\in\mu_{q+1}$. Thus $g(X)$ induces a map from $\mu_{q+1}$ into itself. Moreover, $g(X)$ permutes $\mu_{q+1}$ if and only if $g_1(X)$ permutes $\mu_{q+1}$.
\end{proof}

In light of Proposition~\ref{tool}, we can classify permutation polynomials among certain classes of polynomial with at most eight terms.

\begin{thm}\label{pp-8}
Assume $q=p^k$ for some prime $p$ and some integer $k>0$. Suppose $Q,R,S$ are integers of the form $p^i$ for some integer $i\ge 0$. Write $n\colonequals Q+R+S$, and let $r>0$ be an integer with $r\equiv n\pmod{q+1}$. Suppose $v\in\F_q\cup\mu_{q+1}$ with $v^2\ne 1$. Write $D(X)\colonequals (vX+1)^n-v(X+v)^n$ and $B(X)\colonequals a D(X)$ for some $a\in\F_{q^2}^*$. Then $f(X)\colonequals X^r B(X^{q-1})$ permutes\/ $\F_{q^2}$ if and only if $\gcd(r,q-1)=1$, and either
\begin{itemize}
\item $v\in\F_q$ and $\gcd(n,q+1)=1$, or
\item $v\in\mu_{q+1}$ and $\gcd(n,q-1)=1$.
\end{itemize} 
\end{thm}

\begin{proof}
It follows directly from Lemma~\ref{old} and Proposition~\ref{tool}.    
\end{proof}

Now we are ready to show Theorem~\ref{mth}, as a case of Theorem~\ref{pp-8} in which $v$ is a root of unity so that at least three terms vanish.

\begin{proof}[Proof of Theorem~\ref{mth}]
It is easy to check that $b\in\{0,\pm 1\}$ if and only if $m\in\{3,4,6\}$. If $(Q,R,S) \equiv (1,1,1) \pmod m$, then the conclusion follows from Theorem~\ref{pp-8} by taking $a=(v^3-v)^{-1}$. Suppose $(Q,R,S) \equiv (-1,-1,-1) \pmod m$. Write $B_1(X)\colonequals X^{Q+R+S} B(1/X)$, by Lemma~\ref{old} $X^r B(X^{q-1})$ permutes $\F_{q^2}$ if and only if $X^r B_1(X^{q-1})$ permutes $\F_{q^2}$. So the conclusion follows from Theorem~\ref{pp-8} by taking $a=(1-v^{-2})^{-1}$. If $(Q,R,S) \equiv (1,1,-1) \pmod m$, then the conclusion follows from Theorem~\ref{pp-8} by taking $a=(v^2-1)^{-1}$. Suppose $(Q,R,S) \equiv (-1,-1,1) \pmod m$. Write $B_2(X)\colonequals - X^{Q+R+S} B(1/X)$, by Lemma~\ref{old} $X^r B(X^{q-1})$ permutes $\F_{q^2}$ if and only if $X^r B_2(X^{q-1})$ permutes $\F_{q^2}$. So the conclusion follows from Theorem~\ref{pp-8} by taking $a=(v^{-1}-v)^{-1}$. It concludes the proof.
\end{proof}

Similarly, we state as follows explicitly the case of Corollary~\ref{geo-cyc} in which $n=Q+R-S$ for powers $Q,R,S$ of the characteristic of $\F_q$.

\begin{prop}\label{tool'}
Assume $q=p^k$ for some prime $p$ and some integer $k>0$. Suppose $Q,R,S$ are integers of the form $p^i$ for some integer $i\ge 0$. Suppose $v\in\F_q\cup\mu_{q+1}$ with $v^2\ne 1$. Write 
\[
D(X)\colonequals (vX+1)^{Q+R}(X+v)^S-v(X+v)^{Q+R}(vX+1)^S,
\]
so that we have explicitly 
\begin{align*}
D(X) = &\,
(v^{Q+R}-v^{S+1}) X^{Q+R+S}
+ (v^{Q+R+S}-v) X^{Q+R} \\
&\,+ (v^Q-v^{R+S+1}) X^{Q+S}
+ (v^R-v^{Q+S+1}) X^{R+S} \\
&\,+ (v^{Q+S}-v^{R+1}) X^Q
+ (v^{R+S}-v^{Q+1}) X^R \\
&\,+ (1-v^{Q+R+S+1}) X^S
+ (v^S-v^{Q+R+1}).
\end{align*}
Write $N(X)\colonequals X^{Q+R+S} D(1/X)$. Then $g(X)\colonequals N(X)/D(X)$ maps $\mu_{q+1}$ into itself. Moreover, $D(X)$ has no roots in $\mu_{q+1}$ and $g(X)$ permutes $\mu_{q+1}$ if and only if $v\in\F_q$ and $\gcd(Q+R-S,q+1)=1$.
\end{prop}

\begin{proof}
Write $n\colonequals Q+R-S$. The case $n=0$ is easy, and the case $n<0$ is similar to the case $n>0$. Henceforth we suppose $n>0$. Write $\rho(X)\colonequals (vX+1)/(X+v)$ and $\eta(X)\colonequals (-vX+1)/(X-v)$. By Corollary~\ref{geo-cyc} $g_1(X)\colonequals \eta(X)\circ X^n\circ\rho(X)$ induces a map from $\mu_{q+1}$ into itself. It is easy to check that $g_1(X)=N(X)/D(X)$, so that $g_1(X)=g(X)$, hence $g(X)$ maps $\mu_{q+1}$ into itself. Since $\deg(g_1)=n$, we know $\gcd(N,D) = (vX+1)^S(X+v)^S$, which has no roots in $\mu_{q+1}$ if and only if $v\in\F_q$. Since $g_1(X)$ maps  $\mu_{q+1}$ into itself, it follows $D(X)$ has no roots in $\mu_{q+1}$ if and only if $v\in\F_q$. Henceforth we suppose $v\in\F_q$. By Corollary~\ref{geo-cyc}, $\gcd(n,q+1)=1$ if and only if $g_1(X)$ permutes $\mu_{q+1}$, or equivalently $g(X)$ permutes $\mu_{q+1}$. 
\end{proof}

In light of Proposition~\ref{tool'}, we can classify permutation polynomials among certain classes of polynomial with at most eight terms.

\begin{thm}\label{pp-8'}
Assume $q=p^k$ for some prime $p$ and some integer $k>0$. Suppose $Q,R,S$ are integers of the form $p^i$ for some integer $i\ge 0$. Let $r>0$ be an integer with $r\equiv Q+R+S\pmod{q+1}$. Suppose $v\in\F_q\cup\mu_{q+1}$ with $v^2\ne 1$. Pick $a\in\F_{q^2}^*$, and write 
\[
B(X)\colonequals a \bigl( (vX+1)^{Q+R}(X+v)^S-v(X+v)^{Q+R}(vX+1)^S \bigr).
\]
Then $f(X)\colonequals X^r B(X^{q-1})$ permutes\/ $\F_{q^2}$ if and only if $v\in\F_q$ and $\gcd(r,q-1)=1$ and $\gcd(Q+R-S,q+1)=1$.
\end{thm}

\begin{proof}
It follows directly from Lemma~\ref{old} and Proposition~\ref{tool'}.   
\end{proof}

Now we are ready to show Theorem~\ref{mth'}, as a case of Theorem~\ref{pp-8'} in which $v$ is a root of unity so that at least three terms vanish.

\begin{proof}[Proof of Theorem~\ref{mth'}]
It is easy to check that $b\in\{0,\pm 1\}$ if and only if $m\in\{3,4,6\}$. If $(Q,R,S) \equiv (1,1,1) \pmod m$, then the conclusion follows from Theorem~\ref{pp-8'} by taking $a=(v^3-v)^{-1}$. Suppose $(Q,R,S) \equiv (-1,-1,-1) \pmod m$. Write $B_1(X)\colonequals X^{Q+R+S} B(1/X)$, by Lemma~\ref{old} $X^r B(X^{q-1})$ permutes $\F_{q^2}$ if and only if $X^r B_1(X^{q-1})$ permutes $\F_{q^2}$. So the conclusion follows from Theorem~\ref{pp-8'} by taking $a=(1-v^{-2})^{-1}$. Suppose $(Q,R,S) \equiv (1,-1,1) \pmod m$. Write $B_2(X)\colonequals X^{Q+R+S} B(1/X)$, by Lemma~\ref{old} $X^r B(X^{q-1})$ permutes $\F_{q^2}$ if and only if $X^r B_2(X^{q-1})$ permutes $\F_{q^2}$. So the conclusion follows from Theorem~\ref{pp-8'} by taking $a=(v^2-1)^{-1}$. If $(Q,R,S) \equiv (-1,1,-1) \pmod m$, then the conclusion follows from Theorem~\ref{pp-8'} by taking $a=(v-v^{-1})^{-1}$. If $(Q,R,S) \equiv (1,1,-1) \pmod m$, then the conclusion follows from Theorem~\ref{pp-8'} by taking $a=(v^2-1)^{-1}$. Suppose $(Q,R,S) \equiv (-1,-1,1) \pmod m$. Write $B_3(X)\colonequals - X^{Q+R+S} B(1/X)$, by Lemma~\ref{old} $X^r B(X^{q-1})$ permutes $\F_{q^2}$ if and only if $X^r B_3(X^{q-1})$ permutes $\F_{q^2}$. Thus the conclusion follows from Theorem~\ref{pp-8'} by taking $a=(v^{-1}-v)^{-1}$.
\end{proof}

Finally, we show Theorem~\ref{equiv} as follows.

\begin{proof}[Proof of Theorem~\ref{equiv}]
First we suppose $q\equiv 1\pmod m$. Let $\eta$ be the $\F_q$-vector space automorphism of $\F_{q^2}$ which is induced by $X^q+vX$. Let $\rho$ be $\F_q$-vector space automorphism of $\F_{q^2}$ which is induced by $X^q-vX$ if $(Q,R,S)\equiv (1,1,1)$ or $(1,1,-1)$ or $(-1,1,-1) \pmod m$ and by $-vX^q+X$ if $(Q,R,S)\equiv (-1,-1,-1)$ or $(-1,-1,1)$ or $(1,-1,1)\pmod m$. Let $g\colon\F_{q^2}\to\F_{q^2}$ be the function which is induced by $X^{Q+R+S}$ for Theorem~\ref{mth} and by $X^{Q+R+qS}$ for Theorem~\ref{mth'}. It is routine to verify that the map on $\F_{q^2}$ induced by $f(X)$ equals some nonzero constant in $\F_{q^2}$ times $\rho\circ g\circ\eta$.

Next we suppose $q\equiv -1\pmod m$, so that $v=u^{q-1}$ for some $u\in\F_{q^2}^*$. Since $(u^q)^{q-1} \ne u^{q-1}$, the function $\eta\colon x\mapsto \bigl(ux+(ux)^q,u^qx+(u^qx)^q\bigr)$ gives an $\F_q$-vector space isomorphism $\F_{q^2}\to\F_q\times\F_q$. Since $u^q=uv$ we have $\eta(x) = u(vx^q+x,x^q+vx)$ for any $x\in\F_{q^2}$. Let $\rho\colon\F_q\times\F_q\to\F_{q^2}$ be the $\F_q$-vector space isomorphism  which is defined by $(x,y)\mapsto x-vy$ if $(Q,R,S)\equiv (1,1,1)$ or $(1,1,-1)$ or $(-1,1,-1) \pmod m$ and by $y-vx$ if $(Q,R,S)\equiv (-1,-1,-1)$ or $(-1,-1,1)$ or $(1,-1,1) \pmod m$. Let $g$ be the map on $\F_q\times\F_q$ which is induced by $(X^{Q+R+S},Y^{Q+R+S})$ for Theorem~\ref{mth} and by $(X^{Q+R}Y^S,Y^{Q+R}X^S)$ for Theorem~\ref{mth'}. It is routine to verify that the map on $\F_{q^2}$ induced by $f(X)$ equals a nonzero constant in $\F_{q^2}$ times $\rho\circ g\circ\eta$, so that $f(X)$ is $\F_q$-equivalent to $g$, which concludes the proof.
\end{proof}



\end{document}